\documentclass[reqno]{amsart}
\usepackage[utf8]{inputenc}
\usepackage{amsthm, amsmath, amssymb}
\usepackage{enumerate}
\usepackage{hyperref}
\usepackage{cleveref}

\hypersetup{colorlinks=true,linkcolor=blue,urlcolor=red,pdfpagemode=FullScreen}

\title{Units of hyperelliptic curves over $\F_2$}

\author{Justin Chen}
\address{\rm School of Mathematics, Georgia Institute of Technology, Atlanta, Georgia, USA}
\email{jchen646@gatech.edu}

\author{Vishal Muthuvel}
\address{\rm Columbia University in the City of New York, New York, New York, USA}
\email{vm2696@columbia.edu}

\date{}

\numberwithin{equation}{section}
\newtheorem{theorem}{Theorem}[section]
\newtheorem{proposition}[theorem]{Proposition}
\newtheorem{lemma}[theorem]{Lemma}
\newtheorem{corollary}[theorem]{Corollary}
\newtheorem{conjecture}[theorem]{Conjecture}

\theoremstyle{definition}
\newtheorem{definition}[theorem]{Definition}
\newtheorem{remark}[theorem]{Remark}
\newtheorem{example}[theorem]{Example}
\newtheorem{discussion}[theorem]{Discussion}

\DeclareMathOperator{\N}{\mathbb{N}}
\DeclareMathOperator{\F}{\mathbb{F}}
\DeclareMathOperator{\Z}{\mathbb{Z}}

\DeclareMathOperator{\m}{\mathfrak{m}}
\DeclareMathOperator{\Ass}{Ass}
\DeclareMathOperator{\T}{\mathcal{T}}
\DeclareMathOperator{\Aut}{Aut}


\begin{document}

\vspace*{-1cm}
\maketitle

\begin{abstract}
We study unit groups of rings of the form $\F_2[x,y]/(y^2 + gy + h)$, for $g, h \in \F_2[x]$ -- in particular, the question of (non)triviality of such unit groups.
Up to automorphisms of $\F_2[x,y]$ we classify such rings into 3 distinct types.
For 2 of the types we show that the unit group is always trivial, and conjecture that the unit group is always nontrivial for the 3rd type.
We provide support for this conjecture both theoretically and computationally, via an algorithm that has been used to compute units in large degrees.
\end{abstract}

\section{Introduction}
Let $R$ be a ring (commutative with $1 \ne 0$), with group of units $R^\times$.
A natural question regarding units is the following:
\begin{equation} \label{eq:mainQ} \tag{1}
\text{What are the rings } R \text{ with } R^\times = \{ 1 \}?
\end{equation}
In other words, what rings have trivial unit group?

In general, there is little hope for a full classification of such rings.
In addition, explicit examples are rare.
Nevertheless, for a restricted class of rings one may hope to identify which rings in the class have trivial units.
In this paper, we examine rings of the form 
\[
\F_2[x,y]/(y^2 + gy + h)
\]
where $g, h \in \F_2[x]$.
Geometrically, these are coordinate rings of affine plane curves over $\F_2$ that admit a $2$-to-$1$ map to a line.

We start in \Cref{sec:motivation} by motivating the study of this class of rings with regard to Question \ref{eq:mainQ}.
In \Cref{sec:orbits} we consider the action of automorphisms of the ambient polynomial ring on these curves, and classify minimal representatives for the orbits to be one of 3 types.
We show in \Cref{sec:eqns} that the problem of determining units in these rings is equivalent to representing $1$ by a binary quadratic form over $\F_2[x]$.
Via a simple degree argument, we show that 2 of the 3 types of minimal representatives always have trivial unit group.
This provides a rich source of (previously unknown) examples of rings with trivial unit group.
Finally, we conjecture a full answer to Question \ref{eq:mainQ} for this class of rings, supported by ample computational evidence.

\section{Rings with trivial unit group} \label{sec:motivation}
We begin by identifying some necessary conditions for a ring to have trivial unit group.
This allows us to find meaningful restrictions to place on a class of rings, for the purposes of Question \ref{eq:mainQ}. 
(As we will see, the number $2$ will be of special importance, appearing in various assumptions/conditions.)

To this end, let $R$ be a ring with trivial unit group.
First, note that $R$ has characteristic $2$, as $-1 = 1$ in $R$.
Next, $R$ has trivial Jacobson radical, i.e. the intersection of all maximal ideals of $R$ equals zero.
This follows from an elementwise characterization of the Jacobson radical: $r \in R$ is in the Jacobson radical if and only if $\{ 1 + ar \mid a \in R \} \subseteq R^\times$.
In particular the nilradical is also trivial, i.e. there are no nonzero nilpotent elements (which can also be seen directly from the fact that $1 + r$ is a unit for any nilpotent $r$).

Thus $R$ is a reduced $\F_2$-algebra, i.e. $R = \F_2[X]/I$, for some set of indeterminates $X$ and radical ideal $I \subseteq \F_2[X]$.
To proceed, we now consider restrictions on $R$, which will take the form of hypotheses on $|X|$ and $I$.
For $|X|$: we assume that the number of indeterminates is $2$, which is the smallest case that does not admit an immediate characterization (if $X = \{x\}$, then the only possibilities for trivial unit group are $I = 0, (x), (x + 1)$).
We may then write $R = \F_2[x,y]/I$ for some radical ideal $I \subseteq \F_2[x,y]$.

Since $\F_2[x,y]$ has dimension 2, the codimension of $I$ is either $0, 1$, or $2$.
If $I$ has codimension $0$, then $I = 0$, and $\F_2[x,y]/I = \F_2[x,y]$ has trivial unit group.
We next treat the case of codimension $2$, i.e. maximal ideals:

\begin{proposition} \label{prop:maxIdeals}
Set $\T := \{ ( x,y),( x+1,y),( x,y+1), ( x+1,y+1) \}$, and let $I, J \subseteq \F_2[x,y]$ be radical ideals. % namely the $\F_2$-rational points in $\mathbb{A}^2_{\F_2}$.
\begin{enumerate}
\item $\T$ is precisely the set of maximal ideals $\mathfrak{m} \subseteq \mathbb{F}_2[x,y]$ for which $\mathbb{F}_2[x,y]/\mathfrak{m}$ has trivial unit group.
\item If there is a maximal ideal $\m \in \Ass(\F_2[x,y]/I) \setminus \T$, then $(\F_2[x,y]/I)^\times \ne \{ 1 \}$.
\item Suppose $\Ass(\F_2[x,y]/I) \cap \T = \emptyset$ and $\Ass(\F_2[x,y]/J) \subseteq \T$. Then $$(\F_2[x,y]/I)^\times = \{ 1 \} \iff (\F_2[x,y]/(I \cap J))^\times = \{ 1 \}.$$
% no element of $\T$ is an associated prime of $I$, and every associated prime of $J$ is an element of $\T$,
\end{enumerate}
\end{proposition}
\begin{proof}
If $\m \subseteq \F_2[x,y]$ is maximal, then $\F_2[x,y]/\m$ is a field, and the only field with trivial unit group is $\F_2$, which proves (1).
Next, suppose $\m \in \Ass(\F_2[x,y]/I)$ is maximal. Since $I$ is radical, we may write $I = I' \cap \m$, for some radical ideal $I' \not \subseteq \m$. Then $I' + \m = \F_2[x,y]$, i.e. $I'$ and $\m$ are coprime, so by Chinese Remainder
\[
\F_2[x,y]/I = \F_2[x,y]/(I' \cap \m) \cong \F_2[x,y]/I' \times \F_2[x,y]/\m
\]
and so $(\F_2[x,y]/I)^\times \cong (\F_2[x,y]/I')^\times \times (\F_2[x,y]/\m)^\times$. This shows that $\m \not \in \T \implies (\F_2[x,y]/I)^\times \ne \{ 1 \}$ which is (2), and if $\m \in \T$, then $(\F_2[x,y]/I)^\times \cong (\F_2[x,y]/I')^\times$, from which (3) follows.
\end{proof}

By \Cref{prop:maxIdeals}, if $I \subseteq \F_2[x,y]$ has trivial unit group, then the only possible associated primes which are maximal ideals are elements of $\T$.
Furthermore, in this case we may consider the pure codimension $1$ part of $I$ (i.e. ``remove'' all elements of $\T$ from the associated primes of $I$, e.g. by saturating $I$ by all elements of $\T$), without affecting the unit group.
Thus without loss of generality, $I$ has pure codimension $1$, i.e. all associated primes of $I$ have codimension $1$.

Now $I$, being a radical ideal in a Noetherian ring, is a finite intersection of prime ideals.
As $\F_2[x,y]$ is a UFD, every codimension $1$ prime ideal is principal and generated by an irreducible element, so
\[
I = (f_1) \cap (f_2) \cap \ldots \cap (f_m) = (f_1 f_2 \cdots f_m)
\]
for some irreducible $f_i \in \F_2[x,y]$ (so $I$ is also principal).
This gives a canonical injection of rings
\[
\F_2[x,y]/I \hookrightarrow \prod_{i=1}^n \F_2[x,y]/(f_i)
\]
which induces an injection on unit groups
\[
(\F_2[x,y]/I)^\times \hookrightarrow \left( \prod_{i=1}^n \F_2[x,y]/(f_i) \right)^\times = \prod_{i=1}^n (\F_2[x,y]/(f_i))^\times \]
so in particular, $(\F_2[x,y]/(f_i))^\times = \{ 1 \}$ for all $i \implies (\F_2[x,y]/I)^\times = \{ 1 \}$.
For simplicity, we therefore restrict our attention to the case that $I = (f_1)$ is prime.
Geometrically, $I$ is the ideal of an irreducible curve in $\mathbb{A}^2_{\F_2}$, the affine plane over $\F_2$.

\begin{remark}
It is perhaps appropriate to mention here why the affine (as opposed to projective) case is of interest.
Quite generally, if $R = \bigoplus_{i \ge 0} R_i$ is an $\N$-graded reduced ring, then every unit of $R$ is homogeneous of degree $0$, i.e. $R^\times = R_0^\times$ (this may be seen by reduction to the domain case, since minimal primes of a graded ring are graded, see also \cite[Proposition 5.3]{Che}).
In particular, if $\F_2[X]$ is a positively graded polynomial ring over $\F_2$ (i.e. $\deg x > 0$ for all $x \in X$), then any ring of the form $\F_2[X]/I$, where $I$ is a \emph{graded} radical ideal, automatically has trivial unit group.

This reasoning applies even to non-standard gradings: for example, with the non-standard grading $\deg(x) = 2$, $\deg(y) = 3$, the ring $\F_2[x,y]/(y^2 - x^3)$ is graded and reduced, hence has trivial unit group.
\end{remark}

Returning to our series of reductions: let $R = \F_2[x,y]/(f)$, for some $f \in \F_2[x,y]$ irreducible.
By switching variables if necessary, we may assume $\deg_y f \le \deg_x f$.
Our final, strongest, hypothesis is that $f$ is monic in $y$ of degree $2$, i.e. $f = y^2 + gy + h$ for some $g, h \in \F_2[x]$.
Equivalently, $\F_2[x] \subseteq R$ is a Noether normalization realizing $R$ as a free $\F_2[x]$-module of rank $2$.
In particular, we have the following normal form for elements in $R$:
\begin{equation} \label{eq:normalForm}
R = \F_2[x,y]/(y^2 + gy + h) = \{ a + by \mid a, b \in \F_2[x], y^2 = gy + h \}
\end{equation}
It is this class of rings that we now study in detail.

\section{Orbits under automorphisms of \texorpdfstring{$\mathbb{F}_2[x,y]$}{F2[x,y]}} \label{sec:orbits}

The parameter space for rings of the form $\F_2[x,y]/(y^2 + gy + h)$ is $\F_2[x] \times \F_2[x]$, whose elements are pairs $(g,h)$ of univariate polynomials over $\F_2$ (not to be confused with the fact that the ring itself is in bijection with $\F_2[x] \times \F_2[x]$).
Our first order of business is to identify a class of \emph{minimal representatives} in this parameter space.
To do this, we consider automorphisms of the ambient polynomial ring $\F_2[x,y]$.

First, note that $\F_2[x]$ has precisely one non-identity ring automorphism, which is the map $\psi : \F_2[x] \to \F_2[x], \; x \mapsto x + 1$ (we may also view this as an automorphism of $\F_2[x,y]$ by setting $\psi(y) = y$).
Next, let $\varphi : \F_2[x,y] \to \F_2[x,y]$ be defined by $\varphi(x) = y, \varphi(y) = x$, and for $f \in \F_2[x]$, define maps of the form
\begin{align} \label{def:phi}
\phi_f : \F_2[x,y] &\to \F_2[x,y] \\
x &\mapsto x \nonumber \\
y &\mapsto y + f. \nonumber
\end{align}
Note that each of $\psi, \varphi$, and $\phi_f$ are involutions, i.e. equal to their own inverse (in fact, $\phi_f \circ \phi_{f'} = \phi_{f + f'}$ for any $f, f' \in \F_2[x]$, so $f \mapsto \phi_f$ is an embedding of the additive group $(\F_2[x], +)$ into $\Aut(\F_2[x,y])$).
It follows from a classic theorem of Van der Kulk \cite{VdK} that the group of ring automorphisms of $\F_2[x,y]$ is generated by $\varphi$ and $\{ \phi_f \mid f \in \F_2[x] \}$ (e.g. $\psi = \varphi \circ \phi_1 \circ \varphi$).

% For our purposes, it suffices to consider automorphisms modulo $\psi$ and $\varphi$, which are precisely the $\phi_f$.
Now the action of $\phi_f$ on $y^2 + gy + h$ is given by
\begin{align*}
y^2 + gy + h &\mapsto (y + f)^2 + g(y + f) + h \\
&= y^2 + gy + (h + gf + f^2)
\end{align*}
which we write symbolically as $\phi_f \cdot (g, h) = (g, h + gf + f^2)$.
This defines a group action of $(\F_2[x], +)$ on the parameter space $\F_2[x] \times \F_2[x]$, for which we may consider the orbits.

\begin{definition}
Let $(g, h) \in \F_2[x] \times \F_2[x]$.
We say that $(g, h)$ is a \emph{minimal representative} if $\deg h$ is minimal in its orbit under $(\F_2[x], +)$, i.e. for any $f \in \F_2[x]$, $\deg(h) \le \deg(h + gf + f^2)$. (We use the convention $\deg(0) := -\infty$.)
\end{definition}

Since $g$ is constant over all representatives of an orbit, we may equivalently characterize minimal representatives as pairs $(g, h)$ such that:
\begin{itemize}
\item $\deg g + \deg h$ is minimal in the orbit, or
\item $(\deg g, \deg h)$ is lexicographically least in the orbit
\end{itemize}
We can now give a classification of minimal representatives:

\begin{theorem} \label{thm:minRep}
Let $(g,h) \in \F_2[x] \times \F_2[x]$.
Then $(g,h)$ is a minimal representative if and only if one of the following conditions holds:
\begin{enumerate}
\item $2 \deg g < \deg h$ and $\deg h$ is odd
\item $2 \deg g = \deg h$
\item $\deg g > \deg h$
\end{enumerate}
\end{theorem}

\begin{proof}
First, we show that every minimal representative must satisfy one of the 3 conditions.
So suppose $(g,h)$ is a minimal representative.
If $2 \deg g < \deg h$ and $\deg h$ is even, then for any $f \in \F_2[x]$ with $\deg f = \frac{1}{2} \deg h$, we have $\deg(fg) < \deg(f^2) = \deg h$ and hence $\deg(h + fg + f^2) < \deg h$, contradiction.
Similarly, if $\deg g \le \deg h < 2 \deg g$ (which implies $\deg g > 0$), then for any $f \in \F_2[x]$ with $\deg f = \deg h - \deg g$, we have $\deg(f^2) < \deg(fg) = \deg h$, and again $\deg(h + fg + f^2) < \deg h$, contradiction.

Next, we show that each of the 3 cases is indeed a minimal representative.
For case (1): let $2 \deg g < \deg h$ with $\deg h$ odd.
We show that $\deg h \le \deg(h + fg + f^2)$ for any $f \in \F_2[x]$.
If $2 \deg f < \deg h$ then $\deg(h + fg + f^2) = \deg h$, while if $2 \deg f > \deg h$ then $\deg(f^2) > \max\{\deg(h), \deg(fg)\}$, so $\deg(h + fg + f^2) > \deg h$.

For case (2): let $2 \deg g = \deg h$.
Then as in case (1), $2 \deg f > \deg h \implies \deg(h + fg + f^2) > \deg h$, while $2 \deg f \le \deg h \implies \deg(h + fg + f^2) = \deg h$ (note that when $2 \deg f = \deg h$, all of $h, fg, f^2$ have the same degree, hence their sum -- which has an odd number of top degree terms -- does as well).

Finally, for case (3): let $\deg g > \deg h$.
We claim that the only possible $f \in \F_2[x]$ with $\deg(h + fg + f^2) < \deg g$ are $f = 0$ or $f = g$.
To see this, note that 
% suppose $f(g + f) \ne 0$, and set $h' := h + fg + f^2$.
% Then 
$f(g + f) = fg + f^2 = h + (h + fg + f^2)$ has degree $< \deg g$.
However, assuming any of ($\deg f > \deg g$), ($\deg f = \deg g$ and $f \ne g$), or ($\deg f < \deg g$ and $f \ne 0$) implies that $\deg(f(g + f)) \ge \deg g$, a contradiction.
This shows that for any $f \in \F_2[x]$, either $h + fg + f^2 = h$, or $\deg h < \deg g \le \deg(h + fg + f^2)$.
\end{proof}

We formalize the result of \Cref{thm:minRep} into a definition:

\begin{definition}
We say that a minimal representative $(g,h)$ is
\begin{itemize}
\item \emph{Type 1}: if $2 \deg g < \deg h$ and $\deg h$ is odd
\item \emph{Type 2}: if $2 \deg g = \deg h$
\item \emph{Type 3}: if $\deg g > \deg h$.
\end{itemize}
\end{definition}

Although it is possible for a given orbit to have more than one minimal representative, it follows from the proof of \Cref{thm:minRep} that
this can only happen for Types 1 and 2:

\begin{corollary}
Type 3 minimal representatives are unique in their orbit.
\end{corollary}
% \begin{proof}
% If two type 3 minimal representatives $(g, h)$, $(g', h')$ belonged to the same orbit, then $g = g'$ and $h' = h + fg + f^2$ for some $f$.
% Then $f(g + f) = fg + f^2 = h + h'$ has degree $< \deg g$, but this can only happen if $f = 0$ or $f = g$ (i.e. assuming $f(g + f) \ne 0$, both $\deg f \ne \deg g$ and $\deg f = \deg g$ imply $\deg(f(g + f)) \ge \deg g$), so in fact $h + h' = 0$, i.e. $h = h'$.
% \end{proof}

\begin{remark} \label{rem:alg}
From the proof of \Cref{thm:minRep}, we also obtain an algorithm to compute a minimal representative for any orbit: given $(g, h) \in \F_2[x] \times \F_2[x]$ which is not a minimal representative, set $f := x^d$, where 
\[d = \begin{cases}
\frac{\deg h}{2}, &\quad 2 \deg g < \deg h \\
\deg h - \deg g, &\quad \text{else}
\end{cases}\]
(note that in the first case, $\deg h$ is even, and in the second case, $\deg h \ge \deg g$).
Then $\deg(h + fg + f^2) < \deg h$, and iterating this process leads to a minimal representative.
In fact, each iteration removes the top degree monomial of $h$, thus arrives at the minimal representative after at most $1 + \deg h$ steps.
By recording the degree $d$ used at each step, this also constructs term-by-term the polynomial $f$ such that $\phi_f \cdot (g, h)$ is a minimal representative.
\end{remark}

Since the $\phi_f$ are ring automorphisms, there is an isomorphism of unit groups 
\[
(\F_2[x,y]/(y^2 + gy + h))^\times \cong (\F_2[x,y]/(y^2 + gy + h + fg + f^2))^\times
\]
for any $f \in \F_2[x]$.
Thus for our purposes, we may restrict our attention to minimal representatives $(g, h)$ as necessary.

\section{Units of \texorpdfstring{$\F_2[x,y]/(y^2 + gy + h)$}{F2[x,y]/(y2+gy+h)}} \label{sec:eqns}

\subsection{Units to quadratic forms} \label{ssec:quad}
Given the ring $R = \F_2[x,y]/(y^2 + gy + h)$, we now show that the problem of finding units is equivalent to finding solutions to an associated quadratic equation over $\F_2[x]$.
Nearly all subsequent results in this paper will build on this foundational link.
Recall from \Cref{eq:normalForm} that elements of $R$ have a normal form $\{ a + by \mid a, b \in \F_2[x] \}$, so we may identify the element $a + by \in R$ with the pair $(a, b) \in \F_2[x] \times \F_2[x]$ (not to be confused with the ideal in $\F_2[x]$ generated by $a, b$, which we will always denote by $\gcd(a, b)$).

\begin{proposition} \label{prop:fundEq}
Let $g, h \in \F_2[x]$.
An element $a + by \in \F_2[x,y]/(y^2 + gy + h)$ is a unit if and only if $(a,b) \in \F_2[x] \times \F_2[x]$ satisfies the quadratic equation
\begin{equation} \label{eq:fundEq}
a^2 + abg + b^2h = 1
\end{equation}
In this case, the inverse unit is given by $(a + by)^{-1} = (a + bg) + by$.
\end{proposition}

\begin{proof}
Recall that $R := \F_2[x,y]/(y^2 + gy + h)$ is a free $\F_2[x]$-module with basis $\{ 1, y \}$.
An element $a + by \in R$ is a unit if and only if there exists $c + dy \in R$ with
\[
(a + by)(c + dy) = 1.
\]
% \begin{align*}
% 1 &= (a + by)(c + dy) \\
% &= ac + (ad + bc)y + bdy^2 \\
% &= ac + (ad + bc)y + bd(gy + h) \\
% &= ac + bdh + (ad + bc + bdg)y.
% \end{align*}
Expanding and equating coefficients gives the system of equations (over $\F_2[x]$)
\begin{align*}
ac + bdh &= 1 \\
ad + bc + bdg &= 0.
\end{align*}
Now the first equation implies $\gcd(a, b) = \gcd(c, d) = 1$ in $\F_2[x]$, so writing the second equation as
\[
ad = b(c + dg)
\]
we see that $b \mid ad$, $\gcd(a, b) = 1 \implies b \mid d$.
Furthermore $\gcd(d, c + dg) = \gcd(d, c) = 1$, hence $d \mid b$.
Thus $d$ and $b$ are the same up to units of $\F_2[x]$, so in fact $d = b$ (as $\F_2[x]^\times = \{ 1 \}$), and similarly $a = c + dg = c + bg$.
This gives the desired form of the inverse unit, and substituting $c = a + bg$, $d = b$ into $ac + bdh = 1$ gives \Cref{eq:fundEq}.
Conversely, any pair $(a, b)$ satisfying $a^2 + abg + b^2h = 1$ gives a unit $a + by$ with inverse $(a + bg) + by$, by the same reasoning.
\end{proof}

In other words: given $g, h \in \F_2[x]$, consider the binary quadratic form over $\F_2[x]$
\begin{equation} \label{eq:quadratic}
Q(X, Y) := X^2 + gXY + hY^2.
\end{equation}
Then \Cref{prop:fundEq} states that units of $\F_2[x,y]/(y^2 + gy + h)$ are in bijection with \emph{representations of} $1$ \emph{by} $Q$.

\begin{remark}
The trivial unit $1$ corresponds to $(a,b) = (1, 0)$, which we call the \emph{trivial solution} to \Cref{eq:fundEq}.
Note that setting $b = 0$ in \Cref{eq:fundEq} forces $a = 1$, so that a solution $(a,b)$ is nontrivial if and only if $b \ne 0$.
\end{remark}

One may ask whether representations of other polynomials by $Q$ gives information about representations of $1$ by $Q$, and this is indeed the case (which, as we will see, relies on the fact that $1$ is a square).
First, we introduce some notation:

\begin{definition} \label{def:sqpsqf}
For $0 \ne f \in \F_2[x]$, define $f_r$, $f_s \in \F_2[x]$ to be the unique polynomials such that $f = f_r f_s^2$ and $f_r$ is squarefree.
Explicitly, if $f = \prod_i p_i^{e_i}$ is the unique prime factorization of $f$, where $p_i \in \F_2[x]$ are irreducible, then $f_r = \prod_i p_i^{e_i \!\pmod 2}$, $f_s = \prod_i p_i^{\lfloor \frac{e_i}{2} \rfloor}$.
\end{definition}

\begin{lemma} \label{lem:div}
Let $f, g \in \F_2[x]$.
Then $f \mid g^2 \iff f_r f_s \mid g$.
\end{lemma}
\begin{proof}
For $\Leftarrow$: observe that $f \mid (f_r f_s)^2$.
For $\Rightarrow$: observe that $f_s^2 \mid g^2 \implies f_s \mid g$, hence $f_r = \frac{f}{f_s^2} \mid (\frac{g}{f_s})^2$.
Since $f_r$ is squarefree, this implies $f_r \mid \frac{g}{f_s}$, i.e. $f_r f_s \mid g$.
\end{proof}

\begin{proposition} \label{prop:fundEq2}
Let $Q$ be as in \Cref{eq:quadratic}.
\begin{enumerate}
\item If $a, b \in \F_2[x]$ with $Q(a,b) = 1$, then $b \mid (1 + a)^2$, $b_r \mid g$, and 
\[
Q \left( \frac{1 + a}{b_rb_s}, b_s \right) = \frac{g}{b_r}.
\]
\item Conversely, if $c, d, e \in \F_2[x]$ with $e$ squarefree, $e \mid g$ and $Q(c,d) = \frac{g}{e}$, then $Q(1 + cde, ed^2) = 1$.
\end{enumerate}
\end{proposition}
\begin{proof}
% From $Q(a,b) = 1$ we see that $b(ag + bh) = abg + b^2h = 1 + a^2 = (1 + a)^2$, hence $b \mid (1 + a)^2$.
Writing $Q(a, b) = 1$ in the form
\[
(a + 1)^2 + (a + 1)bg + b^2h = bg
\]
shows that $b \mid (1 + a)^2$, so it follows from \Cref{lem:div} that $b_r b_s \mid (1 + a)$, i.e. we may write $a + 1 = c b_r b_s$ for some $c \in \F_2[x]$.
Substituting back yields
\[
(c b_r b_s)^2 + (c b_r b_s)(b_r b_s^2)g + (b_r b_s^2)^2 h = (b_r b_s^2)g
\]
and upon cancelling $b_r b_s^2$,
\[
b_r (c^2 + c b_s g + b_s^2 h) = g
\]
so $b_r \mid g$, and dividing through by $b_r$ gives $Q(c, b_s) = \frac{g}{b_r}$.
This proves (1), and tracing these steps backwards gives (2).
\end{proof}

\subsection{Degree relations} \label{ssec:deg}
In this subsection, we consider the constraints that \Cref{eq:fundEq} imposes on the degrees of $a, b, g, h \in \F_2[x]$.
Together with the results of \Cref{sec:orbits}, this leads to a resolution of Question \ref{eq:mainQ} for minimal representatives of Types 1 and 2.

% \begin{lemma} \label{lem:degRels}
% Let $a, b, g, h \in \F_2[x]$ be such that $b \ne 0$ and $a^2 + abg + b^2h = 1$.
% Then:
% \begin{enumerate}
% \item $2 \deg g = \deg h \iff g = 0$ or $g = 1$
% \item $2 \deg g < \deg h \iff \deg h$ is even and $\deg a = \deg b + \frac{1}{2} \deg h$.
% \end{enumerate}
% \end{lemma}
% \begin{proof}
% Note that $b \ne 0$ is equivalent to $\deg b \ge 0$.
% If $h = 0$, then $a(a + bg) = 1 \implies a = a + bg = 1 \implies bg = 0 \implies g = 0$, and in this case both (1) and (2) hold.
% Thus we may assume $h \ne 0$.
% If $b = h = 1$, then $a^2 + ag = 0 \implies a = 0$ or $a = g$
% so that $\deg(b^2h) \ge 0$.
% There is an inequality
% \begin{equation} \label{eq:degIneq}
% \deg(1 + a^2 + b^2h) \le \max \{ 2 \deg a, 2 \deg b + \deg h \}
% \end{equation}
% which is an equality if and only if $\deg a \ne \deg b + \frac{1}{2} \deg h$.
% \end{proof}

\begin{theorem} \label{thm:type12}
Let $(g,h)$ be a minimal representative.
If $(g,h)$ is Type 1, or $(g,h)$ is Type 2 and $\deg g > 0$, then $(\F_2[x,y]/(y^2 + gy + h))^\times = \{ 1 \}$.
\end{theorem}
\begin{proof}
Suppose that $(a, b)$ is a nontrivial unit, i.e. $b \ne 0$.
Note that $\deg h > 0$ under either hypothesis, hence $\deg(b^2h) = \deg(b^2h + 1) > 0$, and furthermore $a \ne 0$.
% We first treat the case that $(g, h)$ is Type 1.
% If $h = 0$, then $a(a + bg) = 1 \implies a = a + bg = 1 \implies bg = 0 \implies g = 0$, and in this case both (1) and (2) hold.
% Thus we may assume $h \ne 0$.
% If $b = h = 1$, then $a^2 + ag = 0 \implies a = 0$ or $a = g$
% so that $\deg(b^2h) \ge 0$.
% There is an inequality
Then
\begin{align} \label{eq:degIneq}
\deg a + \deg b + \deg g &= \deg(abg) \nonumber \\
&= \deg(a^2 + b^2h + 1) \nonumber \\
&\le \max \{ 2 \deg a, 2 \deg b + \deg h \}
\end{align}
and the inequality \ref{eq:degIneq} is an equality if and only if 
\begin{equation} \label{eq:ineq}
2 \deg a \ne 2 \deg b + \deg h.
\end{equation}
% In particular, substituting the cases where \ref{eq:ineq} holds into \ref{eq:degIneq} gives
We now analyze the possiblities for \ref{eq:degIneq}: first,
\begin{align*}
2 \deg a > 2 \deg b + \deg h &\implies \deg a = \deg b + \deg g \\
&\implies 2(\deg b + \deg g) > 2 \deg b + \deg h \\
&\implies 2 \deg g > \deg h
\end{align*}
(note that since $b \ne 0$, subtracting $\deg b < \infty$ is a valid operation). Next,
\begin{align*}
2 \deg a < 2 \deg b + \deg h &\implies \deg a = \deg b + \deg h - \deg g \\
&\implies 2(\deg b + \deg h - \deg g) < 2 \deg b + \deg h \\
&\implies 2 \deg g > \deg h
\end{align*}
(in this case, $a^2 + b^2h + 1 \ne 0 \implies g \ne 0$, so we may add/subtract $\deg g$).
Finally, if $2 \deg a = 2 \deg b + \deg h$, then $\deg h$ is even, and \ref{eq:ineq} does \emph{not} hold, i.e. \ref{eq:degIneq} is a strict inequality, so
\begin{align*}
2 \deg a = 2 \deg b + \deg h &\implies \deg a + \deg b + \deg g < 2 \deg a \\
&\implies \deg b + \deg g < \deg b + \frac{1}{2} \deg h \\
&\implies 2 \deg g < \deg h.
\end{align*}
Since the hypotheses (namely $(g, h)$ Type 1 or 2) are incompatible with all 3 cases, we conclude that there cannot exist a nontrivial unit.
\end{proof}

We remark that in \Cref{thm:type12}, the condition $\deg g > 0$ in the Type 2 case is necessary; the cases with $\deg g \le 0$ are treated in \Cref{prop:gZero,prop:gOne}.
% $\deg b + \deg h - \deg g$

% $2 \deg a < 2 \deg b + \deg h \implies \deg a = \deg b + \deg h - \deg g$.

% Now suppose $(g, h)$ is Type 1, which implies that \ref{eq:ineq} holds (and hence $g \ne 0$).
% There are 2 cases to consider:
% \begin{itemize}
% \item $2 \deg a > 2 \deg b + \deg h$: by \ref{eq:degIneq}, $\deg a = \deg b + \deg g \implies 2 \deg g > \deg h$, contradicting $(g, h)$ being Type 1.
% \item $2 \deg a < 2 \deg b + \deg h$: by \ref{eq:degIneq}, $\deg a = \deg b + \deg h - \deg g \implies 2 \deg g > \deg h$, contradicting $(g, h)$ being Type 1.
% \end{itemize}
% Next, suppose $(g, h)$ is Type 2 with $\deg g > 0$.
% There are 3 cases to consider:
% \begin{itemize}
% \item $2 \deg a > 2 \deg b + \deg h$: by \ref{eq:degIneq}, $\deg a = \deg b + \deg g \implies 2 \deg g > \deg h$, contradicting $(g, h)$ being Type 2.
% \item $2 \deg a < 2 \deg b + \deg h$: by \ref{eq:degIneq}, $\deg a = \deg b + \deg h - \deg g \implies \deg h < 2 \deg g$, contradicting $(g, h)$ being Type 2.
% \item $2 \deg a = 2 \deg b + \deg h$: here \ref{eq:ineq} does \emph{not} hold, i.e. \ref{eq:degIneq} is a strict inequality.
% \end{itemize}

\begin{proposition} \label{prop:type3degs}
Let $(g, h)$ be a Type 3 minimal representative with $\deg h > 0$.
Then for any nontrivial unit $a + by \in (\F_2[x,y]/(y^2 + gy + h))^\times$, exactly one of the following holds:
\begin{enumerate}[\normalfont(I)]
\item $\deg a = \deg b + \deg h - \deg g$
\item $\deg a = \deg b + \deg g$.
\end{enumerate}
Moreover, a unit satisfies {\normalfont(I)} if and only if its inverse satisfies \normalfont(II).
\end{proposition}
\begin{proof}
Under the assumption $\deg h > 0$, the same reasoning as in the proof of \Cref{thm:type12} shows that any nontrivial unit belongs to case (I) or (II) (and the cases are disjoint since $(g, h)$ is Type 3).
The second statement concerning inverses then follows from \Cref{prop:fundEq}.
\end{proof}

\begin{corollary} \label{cor:degLowerBound}
Let $(g, h)$ be a Type 3 minimal representative with $\deg h > 0$.
If $a + by$ is a nontrivial unit in $(\F_2[x,y]/(y^2 + gy + h))^\times$ belonging to case {\normalfont(I)} in \Cref{prop:type3degs} and $\deg a > 0$, then $\deg b \ge 2(\deg g - \deg h)$.
\end{corollary}
\begin{proof}
This follows from the inequality $\deg b \le 2 \deg a$, which in turn follows from $b \mid (1 + a)^2$ (see \Cref{prop:fundEq2}(1)) and $1 + a \ne 0$.
\end{proof}

\subsection{Fundamental Units}
Thus far, the results in \Cref{ssec:quad,ssec:deg} hold for any (nontrivial) unit $a + by \in (\F_2[x,y]/(y^2 + gy + h))^\times$.
However, there are distinguished nontrivial units, which we now examine more closely:

\begin{definition} \label{def:fundUnit}
Let $g, h \in \F_2[x]$, and let $a + by \in (\F_2[x,y]/(y^2 + gy + h))^\times$ be a nontrivial unit.
We say that $a + by$, or equivalently the pair $(a, b) \in \F_2[x] \times \F_2[x]$, is a \emph{fundamental unit} if $\deg a + \deg b$ is minimal over all nontrivial units.
\end{definition}

In particular, when $(\F_2[x,y]/(y^2 + gy + h))^\times = \{ 1 \}$, no fundamental unit exists.
In this way, nontriviality of the unit group is, by definition, equivalent to existence of a fundamental unit.

\begin{remark} \label{rem:uniquenessFundUnit}
Uniqueness of fundamental units is less trivial, though still true.
One way to see this is as follows: after passing to the algebraic closure $\overline{\F_2}$, a classic result of Rosenlicht \cite{Ros} states that for a finitely generated $k$-domain $R$ with $k = \overline{k}$, the group $R^\times/k^\times$ is finitely generated free abelian, of rank strictly less than the number of divisorial components of the boundary of the projective closure.

In this case, if $C \subseteq \mathbb{A}^2_{\overline{\F_2}}$ is the curve defined by $y^2 + gy + h = 0$ (with $(g, h)$ a Type 3 minimal representative), then the projective closure $\overline{C} \subseteq \mathbb{P}^2_{\overline{\F_2}}$ is defined by $y^2z^{d-2} + \tilde{g}y + \tilde{h}z^{d - \deg h} = 0$, where $z$ is a new variable, $\tilde{g}, \tilde{h}$ are the homogenizations of $g, h$ with respect to $z$, and $d := 1 + \deg g$.
Then the boundary $\partial C := \overline{C} \setminus C$ is defined by the vanishing of $z$ on $\overline{C}$, and consists of at most 2 points (namely $(1:0:0)$ and $(0:1:0)$).
Thus the unit group is free abelian of rank $< 2$, hence is either trivial or isomorphic to $\Z$.
See \cite[Corollary 3.2]{CVZ} for more details.

In particular, we may talk about \emph{the} fundamental unit (if it exists), which is a generator for the entire unit group.
Note that although $\Z$ has 2 distinct generators as a cyclic group, by \Cref{prop:type3degs} we may distinguish the unit satisfying $\deg a = \deg b + \deg h - \deg g$ as fundamental (as opposed to its inverse).
Moreover, via the isomorphism 
\[
(\F_2[x,y]/(y^2 + gy + h))^\times \xrightarrow{\sim} \Z
\]
which sends the fundamental unit to $1 \in \Z$, we may identify units belonging to case (I) in \Cref{prop:type3degs} with the positive integers in $\Z$, and those of case (II) with the negative integers.
See \Cref{prop:hDividesg} for an explicit special case of this.
\end{remark}

We note that there is a simple argument showing that the unit groups under consideration are (almost always) torsionfree:

\begin{proposition} \label{prop:torsion}
Let $R = \F_2[x,y]/(y^2 + gy + h)$ for some $g, h \in \F_2[x]$.
If $g \ne 0, 1$, then $R^\times$ is torsionfree.
\end{proposition}
\begin{proof}
First, we show by induction that for any element $a + by \in R$ and any $n \in \N$,
\begin{equation} \label{eq:power}
\exists a' \in \F_2[x] \;\; \text{ s.t. } \; (a + by)^{2^n} = a' + b^{2^n}g^{2^n-1}y.
\end{equation}
The base case $n = 0$ holds by taking $a' = a$.
Now if \ref{eq:power} holds for some $n \in \N$, then
\begin{align*}
(a + by)^{2^{n+1}} = (a' + b^{2^n}g^{2^n-1}y)^2 = (a')^2 + b^{2^{n+1}}g^{2^{n+1}-2}h + b^{2^{n+1}}g^{2^{n+1}-1}y
\end{align*}
so \ref{eq:power} holds for $n+1$ as well.

Next, if $R^\times$ was not torsionfree, then there exists a nontrivial torsion unit $a + by$ of prime order, say $p$.
We have $p \ne 2$, as $(a + by)^2 = a^2 + b^2h + b^2gy \ne 1$ (since $b^2g \ne 0$).
Thus $p$ is an odd prime $\implies 2^{p-1} \equiv 1 \pmod p$, so by \ref{eq:power} there exists $a' \in \F_2[x]$ with $a + by = (a + by)^{2^{p-1}} = a' + b^{2^{p-1}}g^{2^{p-1}-1}y$.
Equating coefficients of $y$ gives $b = b^{2^{p-1}}g^{2^{p-1}-1} \implies 1 = (bg)^{2^{p-1}-1} \implies g = 1$, a contradiction.
\end{proof}

In addition to being of lowest degree, there are other conditions that the fundamental unit must satisfy, which are extremely useful for computational purposes.
For $p, f \in \F_2[x]$ with $p$ irreducible, let $\nu_p(f)$ denote the multiplicity of $p$ in the prime factorization of $f$, i.e. $\nu_p(f) := \max \{ i \ge 0 : p^i \mid f \}$, which satisfies the \emph{ultrametric inequality} $\nu_p(f + g) \ge \min \{ \nu_p(f), \nu_p(g) \}$, with equality if $\nu_p(f) \ne \nu_p(g)$.

\begin{proposition} \label{prop:fundUnitConstraints}
If $a + by \in (\F_2[x,y]/(y^2 + gy + h))^\times$ is a fundamental unit with $\deg b > 0$, then (with notation as in \Cref{def:sqpsqf}):
\begin{enumerate}
\item $g$ does not divide $b$. Equivalently, there exists an irreducible factor $p$ of $\frac{g}{b_r}$ such that $\nu_p(b_s) < \frac{1}{2} \nu_p \left( \frac{g}{b_r} \right)$.
\item If $g$ is irreducible, then $b$ is a square.
\item If $p$ is an irreducible factor of $g$, and there exist $h_1, h_2 \in \F_2[x]$ such that $p \nmid h_1$ and $h = h_1p + h_2^2$, then $\nu_p(b) \ge \nu_p(g) - 1$.
\end{enumerate}
\end{proposition}

\begin{proof}
(1): Set $g' := \frac{g}{b_r}$.
By \Cref{lem:div}, $g \mid b \iff g' \mid \frac{b}{b_r} = b_s^2 \iff (g')_r(g')_s \mid b_s$.
If this is the case, then setting $c := \frac{1 + a}{b_r b_s}$, by \Cref{prop:fundEq2}(1) we have
\[
Q(c, b_s) = c^2 + cb_s g + b_s^2h = g'
\]
from which we see that $g' \mid c^2 \implies (g')_r (g')_s \mid c$.
% This means that $((g')_r(g')_s)^2$ divides $Q(c, b_s)$
Then $((g')_r(g')_s)^2 \mid Q(c, b_s)$, so
\begin{align*}
((g')_r(g')_s)^2 \left( \Big( \frac{c}{(g')_r(g')_s} \Big)^2 + \frac{c}{(g')_r(g')_s} \frac{b_s}{(g')_r(g')_s} g + \Big( \frac{b_s}{(g')_r(g')_s} \Big)^2 h \right) &= g' \\
\implies (g')_r \left( \Big( \frac{c}{(g')_r(g')_s} \Big)^2 + \frac{c}{(g')_r(g')_s} \frac{b_s}{(g')_r(g')_s} g + \Big( \frac{b_s}{(g')_r(g')_s} \Big)^2 h \right) &= 1
\end{align*}
hence $Q \Big( \dfrac{c}{(g')_r(g')_s}, \dfrac{b_s}{(g')_r(g')_s} \Big) = 1$ gives a unit of strictly smaller degree than the fundamental unit (since $\deg b > 0$), a contradiction.

The second statement in (1) follows since $\nu_p(b_s) \ge \frac{1}{2} \nu_p(g')$ for all prime factors $p$ of $g' \iff g' \mid b_s^2 \iff g \mid b$.

(2): Since $b_r \mid g$ and $g$ is irreducible, we must have either $b_r = g$ (which is ruled out by (1)), or $b_r = 1$, in which case $b = b_s^2$.

(3): The statement is true if $\nu_p(g) = 1$, so we may assume $\nu_p(g) > 1$.
Substituting $h = h_1p + h_2^2$ into \Cref{eq:fundEq} yields
\[
(1 + a + h_2b)^2 = pb \Big( a \Big( \frac{g}{p} \Big) + bh_1 \Big)
\]
which shows that $\nu_p(pb( a ( \frac{g}{p} ) + bh_1)) = 1 + \nu_p(b) + \nu_p(a ( \frac{g}{p} ) + bh_1)$ is even.
If $\nu_p(b) = 0$, then $\nu_p(a ( \frac{g}{p} ) + bh_1) = 0$ (as $\nu_p(a (\frac{g}{p})) > 0$) $\implies 1$ is even, a contradiction.
Thus $\nu_p(b) > 0$, which implies $\nu_p(a) = 0$ (recall that $\gcd(a, b) = 1$).
Now if $\nu_p(b) < \nu_p(g) - 1 = \nu_p(a (\frac{g}{p}))$, then by the ultrametric (in)equality $\nu_p(a ( \frac{g}{p} ) + bh_1) = \nu_p(bh_1) = \nu_p(b)$, hence $1 + 2 \nu_p(b)$ is even, a contradiction.
\end{proof}

\begin{remark} \label{rem:formOfh}
% A few remarks on \Cref{prop:fundUnitConstraints} are in order:
i) The hypothesis on $h$ in \Cref{prop:fundUnitConstraints}(3) is quite mild: since $\F_2[x]/(p)$ is a finite field of characteristic $2$, every element is a square (i.e. the Frobenius map is surjective).
Thus for instance if $p = x$, then $h$ satisfies the hypothesis if and only if the linear coefficient of $h$ ($= \frac{dh}{dx} \big|_{x = 0}$) is nonzero.

ii) In fact, the proof of \Cref{prop:fundUnitConstraints}(3) holds for any nontrivial unit.
However, if $a + by$ is the fundamental unit, and the inequalities $\nu_p(b) \ge \nu_p(g) - 1$ hold for all primes $p$ dividing $g$, then by \Cref{prop:fundUnitConstraints}(1), one of these inequalities must be an equality.
\end{remark}

\section{Special cases}

\subsection{\texorpdfstring{$g, h$}{g, h} constant}
In this subsection we consider the most special cases, where one of $g, h$ is constant, i.e. has degree $-\infty$ or $0$.
We start with the case $g = 0$:

\begin{proposition} \label{prop:gZero}
For any $h \in \F_2[x]$,
\[
(\F_2[x,y]/(y^2 + h))^\times \cong \begin{cases} (\F_2[x], +), &\quad \exists f \in \F_2[x] \text{ with } h = f^2 \\
\{ 1 \}, &\quad \text{else}
\end{cases}
\]
with fundamental unit $(1 + f) + y$ if $h = f^2$.
\end{proposition}
\begin{proof}
If $g = 0$, then $\phi_f(y^2 + h) = y^2 + h + f^2$ (recall that $\phi_f$ is defined in \ref{def:phi}), and the only possible minimal representatives are (i): $(0, h)$ with $\deg h$ odd, or (ii): $(0, 0)$, which occurs if and only if $h = f^2$ for some $f \in \F_2[x]$.

In case (i) (of Type 1): \Cref{eq:fundEq} becomes $a^2 + b^2h = 1 \implies b^2h = (1 + a)^2$, which is impossible if $b \ne 0$ and $\deg h$ is odd, hence the unit group is trivial.

In case (ii) (of Type 2): there is an isomorphism $(\F_2[x], +) \xrightarrow{\sim} (\F_2[x,y]/(y^2))^\times$ given by $r \mapsto 1 + ry$.
Thus $\F_2[x,y]/(y^2)$ has fundamental unit $1 + y$, which is identified with $(1 + f) + y \in \F_2[x,y]/(y^2 + f^2)$ under $\phi_f$.
\end{proof}

In view of \Cref{prop:gZero}, we may implicitly assume in what follows that $g \ne 0$.
Next, we consider $h = 0$:

\begin{proposition} \label{prop:hZero}
For any $0 \ne g \in \F_2[x]$, $(\F_2[x,y]/(y^2 + gy))^\times = \{ 1 \}$.
\end{proposition}
\begin{proof}
For $g \ne 0$, $(y^2 + gy) = (y) \cap (y + g)$ is an intersection of distinct prime ideals, so $(\F_2[x,y]/(y^2 + gy))^\times \hookrightarrow (\F_2[x,y]/(y))^\times \times (\F_2[x,y]/(y+g))^\times \cong \{ 1 \} \times \{ 1 \} \cong \{ 1 \}$.
\end{proof}

The next case of interest is $g = 1$:

\begin{proposition} \label{prop:gOne}
For any $h \in \F_2[x]$,
\[
(\F_2[x,y]/(y^2 + y + h))^\times \cong \begin{cases} \Z \!/3\!\Z, &\quad \exists f \in \F_2[x] \text{ with } h = f^2 + f + 1 \\
\{ 1 \}, &\quad \text{else}
\end{cases}
\]
with fundamental unit $f + y$ if $h = f^2 + f + 1$.
\end{proposition}
\begin{proof}
If $g = 1$, then $\phi_f(y^2 + y + h) = y^2 + y + h + f + f^2$, and the only possible minimal representatives are (i): $(1, h)$ with $\deg h$ odd, (ii): $(1, 1)$, which occurs if and only if $h = f^2 + f + 1$ for some $f \in \F_2[x]$, and (iii): $(1, 0)$ ($\iff h = f^2 + f$ for some $f \in \F_2[x]$).

In cases (i) and (iii) (of Types 1 and 3): the unit group is trivial by \Cref{thm:type12} and \Cref{prop:hZero} respectively.

In case (ii) (of Type 2): we have $\F_2[x,y]/(y^2+y+1) \cong \F_4[x] \implies (\F_2[x,y]/(y^2+y+1))^\times \cong (\F_4[x])^\times \cong \F_4^\times$, which is isomorphic to $\Z\!/3\!\Z$ (and equals $\{ 1, y, y+1 \}$ as a set).
Thus $\F_2[x,y]/(y^2+y+1)$ has fundamental unit $y$, which is identified with $f + y \in \F_2[x,y]/(y^2 + y + f^2 + f + 1)$ under $\phi_f$.
\end{proof}

\Cref{prop:torsion,prop:gZero,prop:gOne} give a complete resolution to the question of possible torsion in $(\F_2[x,y]/(y^2 + gy + h))^\times$.
Note also that applying the algorithm in \Cref{rem:alg} in the case $g = 1$ gives, as a byproduct, an algorithm for determining when a polynomial $h \in \F_2[x]$ is of the form $f^2 + f$ for some $f \in \F_2[x]$ (and computes such an $f$ when it exists).
Finally, from \Cref{prop:gOne} we see that the condition $\deg b > 0$ in \Cref{prop:fundUnitConstraints}(1) is necessary.

The remaining case $h = 1$ is subsumed in the next subsection.
% Finally, we consider $h = 1$:

% \begin{proposition} \label{prop:hOne}
% For any $g \in \F_2[x]$ with $g \ne 0, 1$, $(\F_2[x,y]/(y^2 + gy + 1))^\times \cong \Z$, generated by the fundamental unit $y$ (i.e. $(a,b) = (0, 1)$).
% \end{proposition}
% \begin{proof}
% Note that $a = 0$ is a solution to \Cref{eq:fundEq} if and only if $b = h = 1$.
% \end{proof}

\subsection{\texorpdfstring{$h \mid g$}{h | g}} \label{ssec:hDividesg}
A useful special case is when $h$ divides $g$, since it may be analyzed completely in accordance with \Cref{rem:uniquenessFundUnit}, using only elementary arguments.
We may assume $h \ne g$, since $\phi_1(y^2 + gy + g) = y^2 + gy + 1$ (and $g = 1$ was treated in \Cref{prop:gOne}), so that $(g,h)$ is a Type 3 minimal representative.

\begin{proposition} \label{prop:hDividesg}
Let $R = \F_2[x,y]/(y^2 + gy + h)$ for some $g, h \in \F_2[x]$.
If $(g, h)$ is a Type 3 minimal representative and $h \mid g$, then $R^\times \cong \Z$, with fundamental unit
\[
% (u, v) = 
\begin{cases}
y, \quad &h = 1 \\
1 + \frac{g}{h}y, \quad &\text{ else.}
\end{cases}
\]
\end{proposition}
\begin{proof}
If $(a,b)$ is any nontrivial solution to \Cref{eq:fundEq}, then $a = 0 \implies b^2h = 1 \iff b = h = 1$, and $a = 1 \implies b(g + bh) = 0 \iff g + bh = 0 \implies h \mid g$.
This shows that the fundamental unit exists and is of the desired form (since it follows from \Cref{prop:type3degs} that in the Type 3 case, a nontrivial unit is fundamental if and only if $\deg a$ is minimal).

It remains to show that the unit group is isomorphic to $\Z$, which we do by considering 2 symmetries of $R^\times$: for any unit $a + by \leftrightarrow (a, b)$, define
\begin{align*}
\sigma_1 &: (a, b) \mapsto (a + bg, b) \\
\sigma_2 &: (a, b) \mapsto \begin{cases} 
(b, a) &\quad h = 1 \\
(a, a \frac{g}{h} + b) &\quad \text{else.} \\
\end{cases}
\end{align*}
It is straightforward to check that $\sigma_1, \sigma_2$ are involutions of $(\F_2[x,y]/(y^2 + gy + h))^\times$.
In fact, $\sigma_1$ is inversion, while $\sigma_2$ is inversion followed by multiplication with the fundamental unit, i.e. $\sigma_2(a + by) = (a + by)^{-1}(u + vy)$, where $u + vy$ is the fundamental unit, and the product is taken in $R$.

We now claim that for any nontrivial unit $a + by$, we can always decrease the quantity $\deg a + \deg b$ by repeated applications of $\sigma_1$ and $\sigma_2$:
e.g. if $\deg a > \deg b$ ($\implies \deg a = \deg b + \deg g$ by \Cref{prop:type3degs}), then $\deg(a + bg) < \deg a$; and in the case $h \ne 1$, $\deg a = \deg b + \deg h - \deg g = \deg b - \deg \frac{g}{h} \implies \deg( a \frac{g}{h} + b) < \deg b$, while if $h = 1$ then one can always reduce to the previous case $\deg a > \deg b$ by $\sigma_2$.
Since this process terminates only when $\deg a + \deg b = -\infty$, i.e. at the trivial unit (or the fundamental unit in the case $h = 1$), this shows that every unit is an integer power of the fundamental unit.
Thus the group homomorphism
\begin{align*}
\Z &\to R^\times \\
k &\mapsto (u + vy)^k
\end{align*}
is surjective, and injective by \Cref{prop:torsion}, hence an isomorphism.
\end{proof}

\begin{remark}
i) Under the isomorphism $\Z \cong R^\times$ above, the maps $\sigma_1, \sigma_2$ correspond to the involutions $k \mapsto -k$, $k \mapsto 1 - k$ of $\Z$.

ii) When $h = 1$, the map $(a, b) \mapsto (a, ag + b)$ corresponds to the involution $k \mapsto 2 - k$ (which explains the necessity of the different formula for $\sigma_2$ when $h = 1$).

iii) The necessity of $\deg a > 0$ in \Cref{cor:degLowerBound} can also be seen from \Cref{prop:hDividesg}.

iv) As another consequence, one also has: if $R = \F_2[x,y]/(y^2 + gy + h)$ has a unit of the form $a + y$ (i.e. $b = 1$), then $R^\times \cong \Z$.
For then $h = a^2 + ag + 1$, so
\begin{align*}
\phi_{a}(y^2 + gy + h) &= y^2 + gy + (a^2 + ag + 1) + ag + a^2 \\
&= y^2 + gy + 1
\end{align*}
and the result follows from \Cref{prop:hDividesg}.
\end{remark}

\subsection{\texorpdfstring{$g^2 \mid h$}{g2 | h}}
Another simple case is when $g^2$ divides $h$, for which the unit group is always trivial (even if $(g, h)$ is not a minimal representative):

\begin{proposition} \label{prop:gDividesh2}
Let $g, h \in \F_2[x]$.
If $g^2 \mid h$ and $\deg g > 0$, then $(\F_2[x,y])/(y^2 + gy + h))^\times = \{ 1 \}$.
\end{proposition}
\begin{proof}
We give 2 different proofs.
For the first: write $h = g^2h'$ for some $h' \in \F_2[x]$, and set $b' := bg$.
Then $1 = a^2 + abg + b^2h = a^2 + ab' + (b')^2h'$, which we may interpret as \Cref{eq:fundEq} for $y^2 + y + h'$.
By \Cref{prop:gOne}, if there is a nontrivial unit, we must have $b' = 1$, but this implies $g = 1$, contradicting $\deg g > 0$.

Alternatively, by \Cref{thm:type12} and \Cref{prop:hZero}, it suffices to show that if the minimal representative for $(g, h)$ is Type 3, then it must be $(g, 0)$.
% by \Cref{prop:hZero}, we may assume throughout that $h \ne 0$
Changing notation, it is enough to show that if $(g, h)$ is a Type 3 minimal representative with $h \ne 0$, then there does not exist $f \in \F_2[x]$ such that $g^2$ divides $h + gf + f^2$.
Suppose that such an $f$ exists, and write $f = qg + r$ with $\deg r < \deg g$, so that $g^2 \mid (h + gf + f^2) \iff g^2 \mid (h + gr + r^2)$.
Since $\deg(h + gr + r^2) \le \max \{ \deg h, \deg g + \deg r, 2 \deg r \} < 2 \deg g = \deg(g^2)$, this can only happen if $h + gr + r^2 = 0$, i.e. $h = gr + r^2$.
Then $\deg(r(g + r)) = \deg r + \deg g = \deg h < \deg g$, which forces $r = 0$, so $h = 0$, a contradiction.
\end{proof}

\subsection{Type 3 general case}

The final case remaining is that of Type 3 minimal representatives with $h \ne 0$ (the case $h = 0$ being treated in \Cref{prop:hZero}).
To this end, we conjecture that all Type 3 minimal representatives with $h \ne 0$ have nontrivial unit group:

\begin{conjecture} \label{conj:type3}
Let $g, h \in \F_2[x]$.
If $\deg g > \deg h$ and $h \ne 0$, then $(\F_2[x,y]/(y^2 + gy + h))^\times \ne \{ 1 \}$.
\end{conjecture}

If \Cref{conj:type3} is true, then combining it with \Cref{thm:type12} (and the special cases above) would give a complete answer to Question \ref{eq:mainQ} for rings of the form $\F_2[x,y]/(y^2 + gy + h)$.
One may view \Cref{prop:hDividesg} and (the second proof of) \Cref{prop:gDividesh2} as evidence for \Cref{conj:type3}.
Additionally, we have a large amount of computational evidence supporting \Cref{conj:type3}, which we now explain.

\begin{discussion} \label{disc:alg}
Fix a Type 3 minimal representative $(g, h)$, and let $E$ be the set of all squarefree divisors of $g$, excluding $g$ itself.
To compute a fundamental unit, by \Cref{prop:fundEq2} it suffices to find a minimal degree solution $(c, d)$ to $Q(X, Y) = \frac{g}{e}$ for some $e \in E$ (note that $e \ne g$ by \Cref{prop:fundUnitConstraints}(1)).
Now if there is a solution with $\deg d =: n$, then by \Cref{prop:type3degs}, $\deg c = n + \deg h - \deg g$.
For convenience, we write $\overline{g}, \overline{h}$ for $\deg g, \deg h$ respectively.
Thus fixing a degree $n$, we set
\begin{align*}
d = x^n + \sum_{i=0}^{n-1} d_i x^i, \quad c = x^{n + \overline{h} - \overline{g}} + \sum_{i=0}^{n + \overline{h} - \overline{g} - 1} c_i x^i
\end{align*}
where $c_i, d_i$ are unknowns (with values in $\F_2$).
Then fixing $e \in E$ and substituting the above into $Q(c, d) = \frac{g}{e}$ gives a system of quadratic equations in the $c_i, d_i$ (obtained by equating coefficients of powers of $x$), to which we also impose the equations $c_i^2 = c_i, d_i^2 = d_i$ for all $i$, altogether giving an ideal $J_n$ in a polynomial ring $\F_2[c_0, \ldots, c_{n + \overline{h} - \overline{g} - 1}, d_0, \ldots, d_{n-1}]$.
We then check whether $J_n$ is the unit ideal (via Gr\"obner bases): if so, then we move on to the next $e \in E$, and after exhausting $E$, increment the degree $n$ and repeat.
Otherwise, for the least $n$ such that $J_n \ne (1)$, we have found a fundamental unit.
From this $J_n$ there are various ways to obtain the $c_i, d_i$: one way is to record which of $J + (d_0), J + (d_1), \ldots$ equal the unit ideal (this avoids a potentially expensive primary decomposition/radical computation).
\end{discussion}

The biggest limitation of the procedure in \Cref{disc:alg} is that there is as yet no theoretical guarantee of termination.
Despite this, we have implemented it in Macaulay2 \cite{GS} and used it to successfully compute fundamental units, in degrees far beyond those feasible via brute-force search.
Since the ideals $J_n$ are not defined by linear forms, it is perhaps surprising that an algorithm using Gr\"obner bases can be so efficient, especially for large $n$.
However, the special form of the equations in $J_n$ allow for a heuristic partial linearization, i.e. there are explicit operations that can be performed on the generators of $J_n$ which transform many of them into linear forms.
This seems to negate a good deal of the worst-case complexity of the Gr\"obner basis computation, and indeed we observe experimentally that the runtime appears to be polynomial in the degree $n$.

An unexpected consequence of these computations is that the degree of a fundamental unit $(a, b)$ (by which we mean $\deg b$) is often far larger than the degree of $g$ (and hence $h$).
We illustrate this with some examples:

\begin{example} \label{ex:units}
All examples below are of Type 3 minimal representatives.
\begin{enumerate}[\normalfont(i)]
\item $(g, h) = (x^3, x^2+1)$ has fundamental unit $(x, x^2+1)$.
\item $(g, h) = (x^2, x+1)$ has fundamental unit $(x^2 + x + 1, x^3 + x)$.
\item $(g, h) = (x^3, x+1)$ has fundamental unit $(x^8 + x^7 + x^4 + x^2 + 1, x^{10} + x^6 + x^4 + x^2)$.
\item $(g, h) = (x^3+x^2+1, x^2+x)$ has fundamental unit $(x^{15}+x^{12}+x^{9}+x^{8}+x^{7}+x^{5}+x^{4}+x^{3}+1, x^{16}+x^{6})$.
\end{enumerate}
\end{example}

Note that any particular $g$ appears in only finitely many Type 3 minimal representatives.
This makes it possible to systematically check \Cref{conj:type3} up to a given value of $\deg g$, and indeed we have verified:

\begin{proposition} \label{prop:verify}
\Cref{conj:type3} is true for $\deg g \le 4$.
\end{proposition}

\Cref{table:maxDegUnits} lists, for a given value of $\deg g$, the maximum degree of a fundamental unit over all Type 3 minimal representatives $(g, h)$.
Pairs that achieve the maximum for $\deg g = 2$ and $3$ appear in \Cref{ex:units}, while the maximum for $\deg g = 4$ and $5$ are achieved by $(g, h) = (x^4+x^3+1,x^3+x)$ and $(x^5+x^2+1,x^2+x)$, respectively.

\begin{table}[h]
\centering
\caption{Maximal degree fundamental units}
\begin{tabular}{|c|c|}
\hline
$\deg g$ & $\max \deg b$ \\ \hline
2 & 3 \\
3 & 16 \\
4 & 52 \\
5 & 134 \\ \hline
\end{tabular}
\label{table:maxDegUnits}
\end{table}

Beyond $\deg g = 5$, the degrees of fundamental units continue to proliferate.
A useful test case is the family $(x^m, x+1)$ for $m \ge 2$: from \Cref{prop:fundUnitConstraints} it follows that the fundamental unit satisfies $\nu_x(b) = m - 1$, which along with \Cref{prop:fundEq2}(1) uniquely determines $b_r$.
At the same time, the degrees of fundamental units for this family seem to grow (asymptotically) exponentially with respect to $\deg g$: for $2 \le m \le 7$, the fundamental units of $(x^m, x+1)$ have degrees $3, 10, 33, 36, 79, 378$ (the last of which took $\sim \!\! 1$ hour to compute).

We conclude with some questions:

\begin{enumerate}
\item Given a Type 3 minimal representative $(g, h)$, can one give an (effective) upper bound on the degree of the fundamental unit?
(Note that \Cref{cor:degLowerBound} provides a lower bound, which in the case of \Cref{ex:units}(i) is sharp.)
Such a bound would have both theoretical and practical implications: for the algorithm described in \Cref{disc:alg}, the bulk of the runtime is always spent checking nonexistence in degrees below that of the fundamental unit.

\item When are unit groups of higher degree curves over $\F_2$ trivial?
Although most of the methods in this paper are specific to the rings $\F_2[x,y]/(y^2 + gy + h)$, it is possible that some techniques may generalize.
On a final positive note, we show that rings of the form $\F_2[x,y]/(y^3 + f^3)$, with $0, 1 \ne f \in \F_2[x]$, have trivial unit group: since $(y^3 + f^3) = (y + f) \cap (y^2 + fy + f^2)$ is a decomposition into distinct prime ideals, there is an injection of rings
\begin{align*}
\F_2[x,y]/(y^3 + f^3) &\hookrightarrow \F_2[x,y]/(y + f) \times \F_2[x,y]/(y^2 + fy + f^2) \\
&\cong \F_2[x] \times \F_2[x,y]/(y^2 + fy + f^2)
\end{align*}
and this last ring has trivial unit group by \Cref{prop:gDividesh2}.
\end{enumerate}

\end{document}